\documentclass[12pt]{amsart}
\usepackage{amsthm,amsfonts,amssymb,fullpage}
\usepackage{color}

\input{xy}
\xyoption{all}
\xyoption{matrix}

\theoremstyle{plain}
 \newtheorem{thm}{Theorem}
 \newtheorem{lem}[thm]{Lemma}
 \newtheorem{prop}[thm]{Proposition}
 \newtheorem{cor}[thm]{Corollary} 
 \newtheorem{question}[thm]{Question} 
 \theoremstyle{remark}
 \newtheorem{rem}[thm]{Remark}
\theoremstyle{definition}
 \newtheorem{example}[thm]{Example}
 \newtheorem{defn}[thm]{Definition}

\def\t{\mathrel{\triangleleft}}
\def\tt{\mathrel{\tilde{\triangleleft}}}
\def\bt{\mathrel{\blacktriangleleft}}

\def\sub{\underline}
\def\ot{\otimes}
\def\wh{\widehat}

\def\RR{\mathbb{R}}
\def\ZZ{\mathbb{Z}}
\def\Hom{\mathrm{Hom}}
\def\Im{\mathrm{Im}}
\def\Ker{\mathrm{Ker}}
\def\oB{\overline{B}}
\def\oH{\overline{H}}

\def\oh{\overline{h}}
\def\lDelta{\overleftarrow{\Delta}}
\def\rDelta{\overrightarrow{\Delta}}
\def\ignore#1{}

\def\Id{\mathrm{Id}}
\newcommand{\mop}[1]{\mathop{\hbox {\rm #1} }\nolimits}
\def\Rack{\mathrm{R}}
\def\Quandle{\mathrm{Q}}
\def\Norm{\mathrm{N}}
\def\Deg{\mathrm{D}}
\def\Late{\mathrm{L}}

\definecolor{myred}{rgb}{.545,0,0}

\title{Bialgebraic approach to rack cohomology}
\author{Simon Covez}
\address{LMJL, Universit\'e de Nantes, 2 chemin de la Houssini\`ere, 44322 Nantes, France}
\email{simon.covez@gmail.com}
\author{Marco Farinati}
\address{Dto. de Matem\'atica FCEyN, Universidad de Buenos Aires -- IMAS Conicet, Argentina}
\email{mfarinat@dm.uba.ar}
\author{Victoria Lebed}
\address{Hamilton Mathematics Institute, Trinity College, Dublin 2, Ireland, \emph{and} LMNO -- UMR 6139, Universit\'e de Caen--Normandie, BP 5186, 14032 Caen Cedex, France}
\email{lebed@unicaen.fr}
\author{Dominique Manchon}
\address{LMBP -- UMR 6620,
	Universit\'e Clermont--Auvergne,
	3 place Vasar\'ely, CS 60026,
        63178 Aubi\`ere, France}
        \email{Dominique.Manchon@uca.fr}

\begin{document}
\begin{abstract}
We interpret the complexes defining rack cohomology in terms of certain d.g. bialgebra. This yields elementary algebraic proofs of old and new structural results for this cohomology theory. For instance, we exhibit two explicit homotopies controlling structure defects on the cochain level: one for the commutativity defect of the cup product, and the other one for the ``Zinbielity'' defect of the dendriform structure. We also show that, for a quandle, the cup product on rack cohomology \emph{restricts} to, and the Zinbiel product \emph{descends} to quandle cohomology.  Finally, for rack cohomology with suitable coefficients, we complete the cup product with a compatible coproduct. 
\end{abstract}
\keywords{Rack cohomology, quandle cohomology, 
 differential graded bialgebra, cup product, dendriform algebra, Zinbiel algebra, homotopy.}
\subjclass[2010]{
 20N02, 
 55N35, 
 16T10.} 

\maketitle
\tableofcontents

\section{Introduction}

A \textit{quandle} $X$ is an algebraic structure whose axioms describe what survives from a group when one only looks at the conjugacy operation. Quandles have been intensively studied since the 1982 work D.~Joyce and S.~Matveev \cite{J82, M82}, who showed how to extract powerful knot invariants from them. But the history of quandles goes back much further. Thus \textit{racks}, which slightly generalize quandles, can be traced back to an unpublished 1959 correspondence between J.~Conway and G.~Wraith; keis, or involutive quandles, appeared in 1943 in an article of M.~Takasaki \cite{T43}; and other related structures were mentioned as early as the end of the 19th century. A thorough account, emphasizing on topological aspects through the rack space, can be found in \cite{FRS07}. Other viewpoints and various applications to algebra, topology, and set theory are treated for instance in \cite{DehBook,AG,Kinyon,QuandlesIntro, PrzSurvey}. 

\textit{Rack cohomology} $H_{\Rack}(X)$ was defined by R.~Fenn, C.~Rourke, and B.~Sanderson in 1995 \cite{FRS95}. It strengthened and extended the applications of racks. The \textit{cup product} $\smile$ on $H_{\Rack}(X)$ was first described in topological terms by F.~Clauwens \cite{Cl}. Later S.~Covez proposed a cubical interpretation, which allowed him to refine the cup product to a \textit{dendriform structure} using shuffle combinatorics, and further to a \textit{Zinbiel product} $\underleftarrow{\smile}$ using acyclic models \cite{C1,C2}. This yields in particular the commutativity of~$\smile$.

Rack cohomology is a particular case of the cohomology of set-theoretic solutions to the \textit{Yang--Baxter equation}, as constructed by J.S. Carter, M.~Elhamdadi and M.~Saito \cite{HomologyYB}. This very general cohomology theory still belongs to the cubical context. It thus carries a commutative cup product, explicitly described by M.~Farinati and J. Garc{\'{\i}}a Galofre \cite{FG}. V.~Lebed \cite{L16} gave it two new interpretations: in terms of M.~Rosso's quantum shuffles \cite{R95}, and via graphical calculus. She gave a graphical version of an explicit homotopy on the cochain level, which controls the commutativity defect of~$\smile$.

\medskip
The purpose of this paper is to study a \textit{differential graded  bialgebra} $B(X)$ that is attached to any rack $X$, and governs the algebraic structure of its cohomology. This construction was first unveiled in \cite{FG}. We show that $B(X)$ is graded cocommutative up to an explicit homotopy, which implies the commutativity of the cup product $\smile$ on $H_{\Rack}(X)$. This yields a purely algebraic version of the diagrammatic construction from \cite{L16}. On a quotient $\oB(X)$ of $B(X)$, we refine the coproduct to a d.g. codendriform structure. Even better: this codendriform structure is coZinbiel up to another explicit homotopy, which has not been described before. As a result, the associative commutative cup product on $H_{\Rack}(X)$ stems from a Zinbiel product $\underleftarrow{\smile}$, which coincides with the one from \cite{C2}.  Both the rack cohomology $H_{\Rack}(X)$ and our d.g. bialgebra $B(X)$ can be enriched with coefficients. For finite $X$ and suitably chosen coefficients, we complete $\smile$ with a compatible coassociative coproduct. Here again our bialgebraic interpretation considerably simplifies all verifications. 
 
The rack cohomology of a quandle received particular attention, starting from the work of J.S. Carter et al. \cite{QuandleHom} and R.A. Litherland and S.~Nelson \cite{LiNe}. It is known to split into two parts, called \emph{quandle} and \emph{degenerate}: $H_{\Rack}=H_{\Quandle}\oplus H_{\Deg}$. The degenerate part $H_{\Deg}$ is far from being empty, since it contains (a shifted copy of) $H_{\Quandle}$ as a direct factor: $H_{\Deg}=H_{\Quandle}[1]\oplus H_{\Late}$. However, as an abelian group, it does not carry any new information: as shown by J.~Przytycki and K.~Putyra \cite{PrzPutyra}, it is completely determined by $H_{\Quandle}$. We recover these cohomology decompositions at the bialgebraic level, and show that the cup product on rack cohomology restricts to $H_{\Quandle}$. This result is new to our knowledge. Rather unexpectedly, our proof heavily uses the Zinbiel product $\underleftarrow{\smile}$ refining $\smile$, even though $\underleftarrow{\smile}$ does not \textit{restrict} to $H_{\Quandle}$. What we show is that $\underleftarrow{\smile}$ \textit{induces} a Zinbiel product on $H_{\Quandle}$: for this we need to regard $H_{\Quandle}$ as a quotient rather than a subspace of $H_{\Rack}$. Our results suggest the following questions:
\begin{enumerate}
\item Does $H_{\Deg}$ allow one to reconstruct $H_{\Quandle}$ as a Zinbiel, or at least associative, algebra?
\item In the opposite direction, does $H_{\Quandle}$ determine the whole rack cohomology  $H_{\Rack}$ as a Zinbiel, or at least associative, algebra?
\end{enumerate}


\medskip
\noindent\textbf{Acknowledgements} M. F. is a Research member of Conicet, partially supported by 
PIP 112-200801 -00900,
PICT 2006 00836, UBACyT X051
and MathAmSud 10-math-01 OPECSHA. 
V. L. was partially supported by a Hamilton Research Fellowship (Hamilton Mathematics Institute, Trinity College Dublin).

\section{Cohomology of shelves and racks}\label{S:RackCohom}
We start with recalling the classical complexes defining rack homology and cohomology. They will be given a bialgebraic interpretation in Section~\ref{S:Bialgebra}. The consequences of this interpretation will be explored in the remainder of the paper. 

A \emph{shelf} is a set $X$ together with a binary operation $\t \colon X\times X\to X, (x,y)\mapsto x\t y$ (sometimes denoted  by $x\t y=x^y$) satisfying the self-distributivity axiom 
\begin{equation}\label{E:SD}
(x\t y)\t z=(x\t z)\t (y\t z)
\end{equation}
for all $x,y,z\in X$. In exponential notation, it reads $(x^y)^z=\left(x ^z\right)^{\left(y ^z\right)}$. A shelf is called a \emph{rack} if the maps $-\t y\colon X\to X$ are bijective for all $y\in X$; a \emph{spindle} if $x\t x=x$ for all $x\in X$; and a \emph{quandle} if it is both a rack and a spindle. The fundamental family of examples of quandles is given by groups~$X$ with $x\t y=y ^{-1}xy$.

Define $C_n(X)=\ZZ X^n$ to be the free abelian group with basis $X^n$, and put $C^n(X)=\ZZ^{X^n}\cong \Hom(C_n(X),\ZZ)$. Define the differential $\partial \colon C_{\bullet}(X)\to C_{\bullet-1}(X)$ as the linearization of
\begin{equation}\label{bord-racks}
\partial(x_1\cdots x_n)=
\sum_{i=1}^{n}(-1)^{i-1} \left(x_1\cdots \wh{x_i} \cdots x_n
-x_1^{x_i}\cdots x_{i-1}^{x_i} x_{i+1} \cdots x_n\right),
\end{equation}
where $\wh x_i$ means that $x_i$ was omitted. Here and below we denote by $x_1\cdots x_n$ the element $(x_1,\ldots, x_n)$ of $X^n$. For cohomology, the differential is taken to be $\partial ^*\colon C^{\bullet}(X)\to C ^{\bullet+1}(X)$. These maps are of square zero (by direct computation, or see Remark \ref{rem2} later) and define respectively the \emph{rack\footnote{The terminology is inconsistent here: indeed rack (co)homology was originally defined for racks, and only later was it realized that it works and is interesting, more generally, for shelves. The same goes for the quandle (co)homology of spindles, considered in Section~\ref{S:Q}.} homology $H^{\Rack}(X)$ and cohomology $H_{\Rack}(X)$} of the shelf $X$.

In knot theory, a quandle $Q$ can be used to color arcs of knot diagrams; a coloring rule involving the operation $\t$ is imposed at each crossing. The three quandle axioms are precisely what is needed for the number of $Q$-colorings of a diagram to depend on the underlying knot only. These $Q$-coloring counting invariants can be enhanced by Boltzmann-type weights, computed using a $2$-cocycle of $Q$. Similarly, $n$-cocycles of $Q$ yield invariants of $(n-1)$-dimensional surfaces knotted in $\RR^{n+1}$. Now, together with arcs one can color diagram regions. The colors can be taken from a $Q$-set, and the weights are given by cocycles with coefficients, which we will describe next.

Given a shelf $X$, an \emph{$X$-set} is a set $S$ together with a map $\bt \colon S\times X\to S$ satisfying 
\begin{equation}\label{E:Xset}
(x\bt y)\bt z=(x\bt z)\bt (y\t z)
\end{equation}
for all $x \in S$, $y,z\in X$. The basic examples are
\begin{enumerate}
\item $X$ itself, with $\t$ as the action map $\bt$;
\item any set with the trivial action $x \bt y = x$;
\item the \emph{structure monoid} $M(X)$ of $X$ (denoted simply by $M$ if $X$ is understood), which is a quadratic monoid defined by generators and relations as follows:
\[M(X)=\langle\, X \, : \, yx^y=xy \text{ for all } x,y \in X\, \rangle;\]
here the action map $\bt$ is concatenation in $M$.
\end{enumerate}
An $X$-set can also be seen as a set with an action of the monoid $M(X)$.

More generally, an \emph{$X$-module} is an abelian group $R$ together with a map $\bt \colon R\times X\to R$ (often written exponentially) which is linear in $R$, and obeys relation~\eqref{E:Xset} for all $x \in R$, $y,z\in X$. In other words, it is an $M(X)$-module. The basic examples are the linearization $\ZZ S$ of an $X$-set $S$, or any abelian group with the trivial action.

Take a shelf $X$ and an $X$-module $R$. Take the free $R$-module $C_n(X,R)=R X^n$ with basis $X^n$, and put $C^n(X,R)= \Hom(C_n(X,R),\ZZ)$. The differential $\partial$ on $C_n$ is the linearization of
\begin{equation}\label{bord-racks-coeffs}
\partial(r x_1\cdots x_n)=
\sum_{i=1}^{n}(-1)^{i-1} \left(r x_1\cdots \wh{x_i} \cdots x_n
-r^{x_i} x_1^{x_i}\cdots x_{i-1}^{x_i} x_{i+1} \cdots x_n\right),
\end{equation}
and the differential on $C^n$ is the induced one. Again, these maps are of square zero and define respectively the \emph{rack homology $H^{\Rack}(X,R)$ and cohomology $H_{\Rack}(X,R)$ of the shelf $X$ with coefficients in $R$}. If $R$ is the linearization of an $X$-set $S$, we use notations $C_{\Rack}(X,S)$, $H_{\Rack}(X,S)$ etc. Choosing as $S$ the empty set, one recovers the previous definitions. Another interesting coefficient choice is the \emph{structure algebra} $A(X)$ of $X$ (often denoted simply by $A$), which is the monoid algebra of $M(X)$:
\[A(X)=\ZZ M(X) \cong \ZZ\langle X\rangle\,\big/\,\langle yx^y-xy\, : \, x,y \in X \rangle.\]
Declaring every $x\in X$ group-like, one gets an associative bialgebra structure on $A$. This coefficient choice is universal in the following sense. Any $X$-module $R$ is a right $M(X)$-module, hence a right $A(X)$-module. Then one has an obvious isomorphism of chain complexes
\[C_n(X,R) \cong R \underset{A(X)}{\otimes} C_n(X,M(X)),\]
where $A$ acts on the first factor of $C_n(X,M)\cong A\otimes_{\ZZ} \ZZ X^n$ by multiplication on the left, and the differential acts on the second factor of $R \otimes_{A} C_n(X,M)$.

\section{A d.g. bialgebra associated to a shelf}\label{S:Bialgebra}

The algebraic objects introduced in this section are aimed to yield a simple and explicit description of a differential graded algebra structure on the complex $(C^\bullet(X),\partial^*)$ above, which is commutative, and in fact even Zinbiel, up to explicit homotopies. 

Fix a shelf $X$. All the (bi)algebra structures below will be over $\ZZ$,  and will be (co)unital. Also, the tensor product $A\ot B$ of two graded algebras will always be endowed with the product algebra structure involving the Koszul sign:
\[(a_1 \ot b_1)(a_2 \ot b_2)=(-1)^{|b_1||a_2|}a_1 a_2 \ot b_1 b_2,\]
where $b_1 \in B$ and and $a_2 \in A$ are homogeneous of degree $|b_1|$ and $|a_2|$ respectively. The Koszul sign also appears when $A^*\ot B^*$ acts on $A\ot B$. Similarly, by a (co)derivation on a graded (co)algebra we will always mean a super-(co)derivation, and by commutativity we will mean super-commutativity.


Define $B(X)$ (also denoted by $B$) as the algebra freely generated by two copies of $X$, with the following relations:
\[B(X)=\ZZ\langle x, e_y \, : \, x,y\in X\rangle \,\big/\, \langle yx^y-xy,ye_{x^y}-e_xy \, : \, x,y \in X \rangle.\]
The interest of this construction lies in the rich structure it carries:
\begin{thm}\label{dgb}
For any shelf $X$, $B(X)$ is a differential graded bialgebra and a differential graded $A(X)$-bimodule, where
\begin{itemize}
\item the grading is given by declaring $|e_x|=1$ and $|x|=0$ for all $x\in X$;
\item the differential $d$ is the unique derivation of degree $-1$ determined by
\[
d(e_x)=1-x,\qquad d(x)=0 \qquad \text{ for all } x\in X;
\]
\item the comultiplication $\Delta \colon B \to B\ot B$ and the counit $\varepsilon \colon B \to k$ are defined on the generators by
\begin{align*}
\Delta(e_x)&=e_x\ot x+1\ot e_x,  &\Delta(x)&= x\ot x \qquad \text{ for all } x\in X,\\
\varepsilon(e_x)&=0,  &\varepsilon(x)&=1 \qquad \text{ for all } x\in X,
\end{align*}
and extended multiplicatively; 
\item the $A$-actions $\lambda \colon A \otimes B \to B$ and $\rho \colon B \otimes A \to B$ are defined by  
\[x \cdot b \cdot y= xby \qquad \text{ for all } x,y \in X,\ b \in B.\]
\end{itemize} 
\end{thm}
By differential graded bialgebra we mean that the differential is both a derivation with respect to multiplication, and coderivation with respect to comultiplication.

Notice that $B$ is neither commutative nor cocommutative in general.

As usual, from Theorem~\ref{dgb} one deduces a d.g. algebra and a d.g. $A(X)$-bimodule structures on the graded dual $B^*(X)$ of $B(X)$.

\begin{proof}
Since the relations are homogeneous, $B$ is a graded algebra. 
In order to see that $d$ is well defined, one must check that the relations
$yx^y\sim xy$ and $ye_{x^y}\sim e_xy$ are compatible with $d$. The first relation
is easier:
\[
d(yx^y-xy)=
d(y)x^y+yd(x^y)-d(x)y-xd(y)=0+0-0-0=0.
\]
For the second relation, one has
\[
d(ye_{x^y}-e_xy)=
yd(e_{x^y})-d(e_x)y=
y(1-x^y)-(1-x)y=
y-yx^y-y+xy=xy-yx^y.
\]
So the ideal of relations defining $B$ is stable by $d$.

Since $d$ is a derivation and $d^2$ vanishes on generators, we have $d^2=0$, hence a structure of differential graded algebra on $B$. 

Next, we need to check that $\Delta$ is well defined. The first relation is easy since
all $x\in X$ are group-like in $B$:
\[
\begin{array}{ccl}
\Delta(xy-yx^y) &=&(x\ot x)(y\ot y)-(y\ot y)(x^y\ot x ^y) 
\\ &=& xy\ot xy-yx^y\ot yx^y
\\ &=& xy\ot (xy-yx^y)+(xy-yx^y)\ot yx^y.
\end{array}
\]
For the second relation, we check:
\[
\begin{array}{rcl}
\Delta(ye_{x^y}-e_xy)&
=&
(y\ot y)(e_{x^y}\ot x^y+1\ot e_{x^y})
-(e_x\ot x+1\ot e_x)(y\ot y)\\
&=&
ye_{x^y}\ot yx^y+y\ot ye_{x^y}
-e_xy\ot xy-y\ot e_xy\\
&=&
(ye_{x^y}-e_xy)\ot yx^y
+
e_xy\ot (yx^y-xy)
+y\ot (ye_{x^y}- e_xy).
\end{array}
\]
So the ideal defining the relations is also a coideal.

Clearly, $\Delta$ respects the grading.

Let us now check that $d$ is a coderivation. It is enough to see this on generators:
\begin{eqnarray*}
\Delta(d(x))&=&\Delta(0)=0=d(x)\ot x+x\ot d(x)=(d\ot 1+1\ot d)\Delta (x);\\
\Delta(d(e_x))&=&\Delta(1-x)=1\ot 1-x\ot x,
\end{eqnarray*}
which coincides with
\begin{eqnarray*}
(d\ot 1+1\ot d)\Delta (e_x)&=&
(d\ot 1+1\ot d)(e_x\ot x+1\ot e_x)\\
&=&
(1-x)\ot x+1\ot (1-x)=
1\ot x-x\ot x+1\ot 1-1\ot x\\
&=&
1\ot 1-x\ot x.
\end{eqnarray*}

The map $\varepsilon$ is also well defined, since
\[\varepsilon(xy-yx^y)=\varepsilon(ye_{x^y}-e_xy)=0.\]
An easy verification on the generators shows that it is indeed a counit.

Finally, the formula $x \cdot b \cdot y= xby$ obviously defines commuting degree-preserving $A$-actions on $B$. By the definition of the differential $d$, one has $d(xby)=xd(b)y$, thus $d$ respects this bimodule structure.
\end{proof}

\begin{example}\label{exDelta}
Let us compute $\Delta(e_xe_y)$. By definition, $\Delta(e_xe_y)=\Delta(e_x)\Delta(e_y)$ in $B\ot B$, and this is equal to
\begin{eqnarray*}
(e_x\ot x+1\ot e_x)(e_y\ot y+1\ot e_y)
&=&e_xe_y\ot xy+ e_x\ot xe_y
-e_y\ot e_xy+1\ot e_xe_y\\
&=&e_xe_y\ot xy+ e_x\ot xe_y
-e_y\ot ye_{x^y}+1\ot e_xe_y.
\end{eqnarray*}
Note the Koszul sign appearing in the product $(1\ot e_x)( e_y \ot y) =-e_y\ot e_xy$.
\end{example}

The structure on $B(X)$ survives in homology:
\begin{prop}\label{prop:dgb_hom}
For any shelf $X$, the homology $H(X)$ of $B(X)$ inherits a graded algebra structure. Moreover, the $A(X)$-actions on $B(X)$ induce trivial actions on $H(X)$: we have
\[ x \cdot h \cdot y= h \qquad \text{ for all }x,y \in X,\ h \in H.\]
Dually, the cohomology $H^{\ast}(X)$ of $B(X)^{\ast}$ inherits a graded algebra structure and trivial $A(X)$-actions.
\end{prop}

Here and below by $B(X)^{\ast}$ we mean the graded dual of $B(X)$.

\begin{proof}
The only non-classical statement here is the triviality of the induced actions. Take $x \in X,\ b \in B$. By the definition of the differential $d$, one has $d(e_x b) = d(e_x)b -e_xd(b)$, hence
\begin{equation}\label{E:L1property}
d(e_x b) =(1-x)b -e_xd(b).
\end{equation}
If $b$ is a cycle, this shows that $x \cdot b = b$ modulo a boundary. Hence the induced left $A$-action on $H$ is trivial. The cases of the right action and the actions in cohomology are analogous.
\end{proof}

The proposition implies the following remarkable property of $H$: if $b \in B$ is a representative of some homology class in $H$, and if one lets an $x \in X$ act upon all the letters from $X$ occurring in $b$ (where the action is $y \mapsto y^x$), then one obtains another representative of the same homology class.

One can also define a version of the bialgebra $B(X)$ with coefficients in any unital commutative ring $k$: 
\[B(X,k)=k\langle x, e_y \, : \, x,y\in X\rangle \,\big/\, \langle yx^y-xy,ye_{x^y}-e_xy \, : \, x,y \in X \rangle.\]
In particular, all the tensor products should be taken over~$k$. Theorem~\ref{dgb} and its proof extend verbatim to this setting. For suitable coefficients $k$, one can say even more:

\begin{prop}\label{prop:dgb_hom_bialg}
For any shelf $X$ and any field $k$, the homology $H(X,k)$ of $B(X,k)$ inherits a graded bialgebra structure. If moreover $X$ is finite, then the cohomology $H^{\ast}(X,k)$ of $B(X,k)^{\ast}$ also inherits a graded bialgebra structure.
\end{prop}

This results from the following general observation; it is surely known to specialists, however the authors were unable to find it in the literature.
\begin{lem}
Let $k$ be a field.
\begin{enumerate}
\item If $(C=\oplus C_i,d,\Delta)$ is a $k$-linear d.g. coassociative coalgebra, then $\Delta$ induces a coproduct on the homology $H$ of $(C,d)$.
\item If $(C=\oplus C_i,d,\cdot)$ is a $k$-linear d.g. algebra of finite dimension in each degree, then $\cdot$ induces a coproduct on the cohomology $H^{\ast}$ of $(C^{\ast},d^{\ast})$.
\end{enumerate} 
\end{lem}
\begin{proof}
\begin{enumerate}
\item The relation $\Delta d = (d \otimes \Id+\Id \otimes d)\Delta$ implies that $\Delta$ survives in the quotient $C/\Im(d)$. To restrict it further to $H=\Ker(d)/\Im(d)$, we shall check that
\[\Delta(\Ker(d)) \subseteq \Ker(d) \otimes \Ker(d) + \Im(d) \otimes C + C \otimes \Im(d).\]
Since $k$ is a field, the space $K:=\Ker(d)$ has a complement $L$ in $C$, on which $d$ is injective. Putting $I:=\Im(d)$, one has
\begin{align*}
(d \otimes \Id)(L \otimes L) &= I \otimes L,& (\Id \otimes d)(L \otimes L) &= L \otimes I, \\
(d \otimes \Id+\Id \otimes d) (K \otimes L) &= K \otimes I, & (d \otimes \Id+\Id \otimes d) (L \otimes K) &= I \otimes K,\\
(d \otimes \Id+\Id \otimes d) (K \otimes K) &= 0.
\end{align*}
Moreover, in the first two lines all the maps are bijective. Now, from $(d \otimes \Id+\Id \otimes d)\Delta (K)=\Delta d (K)=0$ and from the disjointness of $L$ and $K$ (and hence $I$), on sees that $\Delta (K)$ cannot have components in $L \otimes L$, and its components from $K \otimes L$ (resp., $L \otimes K$) necessarily lie in $I \otimes L$ (resp., $L \otimes I$).
\item Due to the finite dimension in each degree, the product on $(C,d)$ induces a coproduct on $(C^{\ast},d^{\ast})$, to which we apply the first statement.  \qedhere
\end{enumerate} 
\end{proof}

\begin{rem}
The d.g. bialgebra $B(X)$ admits a variation $B'(X)$, where one adds the inverses $x^{-1}$ of the generators $x \in X$, with $|x^{-1}|=0$, $d(x^{-1})=0$, $\Delta(x^{-1})=x^{-1}\ot x^{-1}$, $\varepsilon(x^{-1})=1$. One obtains a d.g. Hopf algebra, with the antipode defined on the generators by
\[s(x)=x^{-1},\, s(x^{-1})=x, \, s(e_x)=-e_xx^{-1},\]
and extended super-anti-multiplicatively. Indeed, one easily verifies that this map is well defined, of degree $0$, yields the inverse of $\Id$ in the convolution algebra, and commutes with the differential $d$. For the square of the antipode, one computes $s^2(e_x)=xe_xx^{-1}$. In a spindle it equals $e_x$, yielding $s^2=\Id$. In general $s$ need not be of finite order: for the rack $X=\ZZ$ with $x^y=x+1$, one has $s^2\textit{•}(e_{x})=e_{x-1}$. In a rack, one simplifies $s^2(e_x)=e_{x \tt x}$, where the operation $\tt$ is defined by $(x \t y) \tt y = x$ for all $x,y \in X$. The map $x \mapsto x \tt x$ plays an important role in the study of racks; see for instance \cite{Szymik}. Finally, in the computation
\begin{align*}
s(e_{x_1}e_{x_2}\cdots e_{x_{n-1}} e_{x_n}) &= (-1)^{n \choose 2}(- e_{x_n}x_n^{-1})(- e_{x_{n-1}}x_{n-1}^{-1}) \cdots (-e_{x_2}x_2^{-1})(-e_{x_1}x_1^{-1}) \\
&=(-1)^{\frac{n(n+1)}{2}}\, e_{x_n} \, e_{x_{n-1}^{x_n}} \cdots e_{x_2^{x_3\cdots x_n}}\, e_{x_1^{x_2\cdots x_n}} \, x_n^{-1} \cdots x_1^{-1}
\end{align*}
one recognizes the \emph{remarkable map}
\[(x_1,x_2,\ldots,x_n) \mapsto (x_1^{x_2\cdots x_n},x_2^{x_3\cdots x_n},\ldots,x_n)\]
of J.~Przytycki \cite{Prz1}.
\end{rem}

\section{The bialgebra encodes the cohomology}\label{S:BialgebraHom}

We will now show that the d.g. bialgebra $B(X)$ knows everything about the homology $(C_\bullet,\partial)$ and the cohomology $(C^\bullet,\partial^*)$ of our shelf $X$, and about its variations $(C_\bullet^M,\partial)$ and $(C^\bullet_M,\partial^*)$ with coefficients in the structure monoid $M(X)$. 

First, we need to modify $B$ slightly: 
\begin{lem}\label{L:Bquotient}
The following data define a d.g. coalgebra and a right d.g. $A$-module:
\[({\ZZ}\ot_A B,\ \Id_{\ZZ}\ot d,\ \Id_{\ZZ}\ot \Delta,\ \Id_{\ZZ}\ot \varepsilon,\ \Id_{\ZZ}\ot \rho).\]
Here the grading is the one induced from~$B$, and the $A$-action on ${\ZZ}$ is the trivial one: $\lambda \cdot x=\lambda$ for all $x\in X,\ \lambda \in {\ZZ}$.
\end{lem}

The d.g. coalgebra from the lemma will be denoted by $\oB=\oB(X)$. It has obvious variants $\oB(X,k)$ with coefficients in any unital commutative ring~$k$.

\begin{proof}
On the level of abelian groups, one has
\begin{equation}\label{E:BBar}
\oB \simeq B\,\big/\,\langle xb - b\, : \, x \in X,\ b \in B \rangle.
\end{equation} 
The grading survives in this quotient since $|x|=0$. So does the degree $-1$ differential, as $d(x)=0$ and $|x|=0$ imply $d(xb)=xd(b) \sim d(b)$. Further, we have $\Delta(xb)=(x \ot x)\Delta(b) \sim \Delta(b)$, so $\Delta$ induces a coproduct on $\oB$ compatible with the grading and the differential. For the counit, we have $\varepsilon(xb)=\varepsilon(x)\varepsilon(b) =\varepsilon(b)$. Finally, the right $A$-action also descends to $\oB$, as $(xb) \cdot y = xby = x(by) \sim by=b \cdot y$.
\end{proof}

\begin{rem}\label{R:LoseProduct}
Observe that we lose the product in the quotient $\oB$. Indeed, for all $x,y \in X$ we have $y \sim 1$, but $e_x \cdot y = e_xy=ye_{x^y} \sim e_{x^y} \nsim e_x=e_x \cdot 1$.
\end{rem}

\begin{lem}\label{Bvscomplejo}
As a left $A$-module, $B$ can be presented as
\[B\cong A\ot \ZZ\langle X\rangle.\] 
\end{lem}

\begin{proof}
Consider the map 
\begin{align*}
A\ot \ZZ\langle X\rangle &\to B,\\
x_1\cdots x_k \ot y_1 \cdots y_n &\mapsto x_1\cdots x_k e_{y_1} \cdots e_{y_n}.
\end{align*}
It is well defined since the relations in~$A$ hold true for the corresponding generators of~$B$.

Going in the opposite direction is trickier. A monomial $b$ in $B$ is a product of generators of the form $x$ and $e_y$. Let $a(b)$ be what remains in $b$ when all generators of the form $e_y$ are erased. Further, start with a new copy of $b$ and erase all generators of the form $x$ one by one, starting from the left; when erasing a generator $x$, replace all generators of the form $e_y$ to its left by $e_{y^x}$. After that replace all the $e_y$ by $y$. This yields a monomial $t(b) \in \ZZ\langle X\rangle$. Analyzing the defining relations of $B$, and using the self-distributivity axiom \eqref{E:SD} for~$X$, one sees that we obtained a well-defined map 
\begin{align*}
B &\to A\ot \ZZ\langle X\rangle,\\
b &\mapsto a(b)\ot t(b).
\end{align*}

Both maps are clearly $A$-equivariant, and are mutually inverse. 
\end{proof}

From this follows
\begin{prop}\label{Bvscomplejo2}
One has the following isomorphisms of complexes:
\begin{align*}
& (C_\bullet^M,\partial)\cong B, \qquad (C^\bullet_M,\partial^*)\cong B^*,\qquad (C_\bullet,\partial)\cong \oB, \qquad (C^\bullet,\partial^*)\cong \oB^*,\\
& (C_\bullet(X,k M(X)),\partial)\cong B(X,k), \qquad (C^\bullet(X,k M(X)),\partial^*)\cong B^*(X,k),\\
& (C_\bullet(X,k),\partial)\cong \oB(X,k), \qquad (C^\bullet(X,k),\partial^*)\cong \oB^*(X,k).
\end{align*}
Here $B^*$ denotes the graded dual of $B$ with the induced differential, and similarly for $\oB^*$; in the last isomorphisms, the ring $k$ is considered as a trivial $X$-module.
\end{prop}

\begin{proof}
The preceding lemma yields isomorphisms of abelian groups
\[B \cong C_\bullet^M, \qquad \oB \cong \ZZ\langle X\rangle = C_\bullet,\]
and their $k$-versions.

To compute the differential induced on this by~$d$, we use that $d$ is a derivation:
\[
d(e_{x_1}\cdots e_{x_n})
=\sum_{i=1}^ n(-1)^ {i-1} e_{x_1}\cdots e_{x_{i-1}}d(e_{x_i})e_{x_{i+1}}\cdots e_{x_n}\]
\[
=\sum_{i=1}^ n(-1)^ {i-1}e_{x_1}\cdots e_{x_{i-1}}(1-x_i)e_{x_{i+1}}\cdots e_{x_n}\]
\[
=\sum_{i=1}^ n(-1)^ {i-1} e_{x_1}\cdots e_{x_{i-1}}e_{x_{i+1}}\cdots e_{x_n}
-
\sum_{i=1}^ n(-1)^ {i-1}e_{x_1}\cdots e_{x_{i-1}}x_ie_{x_{i+1}}\cdots e_{x_n}.\]
Using the relation $e_xy=ye_{x^y}$, one gets
\[
e_{x_1}\cdots e_{x_{i-1}}x_ie_{x_{i+1}}\cdots e_{x_n}
= x_i e_{x_1^{x_i}}\cdots e_{x_{i-1}^{x_i}}e_{x_{i+1}}\cdots e_{x_n},\]
which in $C_\bullet^M$ corresponds to $x_i x_1^{x_i}\cdots x_{i-1}^{x_i} x_{i+1} \cdots x_n$. We thus recover the differential~\eqref{bord-racks-coeffs}. In the quotient $\oB$, the last computation simplifies as $e_{x_1^{x_i}}\cdots e_{x_{i-1}^{x_i}}e_{x_{i+1}}\cdots e_{x_n}$, and we recover the differential~\eqref{bord-racks}.
\end{proof}

\begin{rem}\label{rem2}
This proposition provides a very simple proof that $\partial^2=0$ in $C_\bullet(X,M(X))$ and its versions.
\end{rem}

From the proof of Lemma~\ref{Bvscomplejo} one deduces the following useful fact:
\begin{lem}\label{L:AinB}
The map $A \to B$, $x \mapsto x$ is an injective algebra morphism.
\end{lem}
In what follows we will often identify $A$ with its isomorphic image in~$B$.

The isomorphisms in Proposition~\ref{Bvscomplejo2} allow one to transport the structure from $B$ and $\oB$ to rack (co)homology:
\begin{thm}\label{thm:main}
Take a shelf $X$ and a field $k$. Then
\begin{enumerate}
\item the chain complex $(C_\bullet(X),\partial)$ carries a coassociative coproduct;
\item the cochain complex $(C^\bullet(X),\partial^*)$ carries an associative product;
\item the chain complex $(C_\bullet(X,M(X)),\partial)$ carries a bialgebra structure;
\item the cochain complex $(C^\bullet(X,M(X)),\partial^*)$ carries an associative product, enriched to a bialgebra structure when $X$ is finite.
\end{enumerate}
This induces
\begin{enumerate}
\item associative products on $H^\ast(X)$, $H^\ast(X,M(X))$, $H^\ast(X,k)$, and $H(X,M(X))$;
\item a coassociative coproduct on $H(X,k)$;
\item a bialgebra structure on $H(X,kM(X))$;
\item an associative product on $H^\ast(X,kM(X))$, which is completed to a bialgebra structure for finite $X$.
\end{enumerate}
\end{thm}

The product in cohomology is called the \emph{cup product}, denoted by~$\smile$.

\begin{example}
Take $f,g\in C^2(X)$. To compute $f\smile g$, one needs to compute
the summands in $\Delta(e_xe_ye_ze_t)$ with two tensors of type $e_{u}$ in each factor. We use the computation from Example~\ref{exDelta}:
\[
\Delta(e_xe_ye_ze_t)=
\Delta(e_xe_y)\Delta(e_ze_t)=\]
\[=
(e_xe_y\ot xy+1\ot e_xe_y+
e_x\ot xe_y
-e_y\ot ye_{x^y})
(e_ze_t\ot zt+1\ot e_ze_t+
e_z\ot ze_t
-e_t\ot te_{z^t})
\]
\[
=e_xe_y\ot xye_ze_t+e_ze_t\ot e_xe_yzt
-e_x e_z\ot   xe_yze_t
+e_xe_t   \ot xe_yte_{z^t}
+e_y e_z\ot ye_{x^y}ze_t
-e_y  e_t\ot ye_{x^y}te_{z^t}+\cdots
\]
where the dots hide terms on which  $f$ and $g$ vanish. Pushing the $e_{u}$'s to the right and the elements of $X$ to the left, we get
\[
e_xe_y\ot xye_ze_t+e_ze_t\ot zx e_{x^{zt}} e_{y^{zt}}
-e_x e_z\ot   xze_{y^z}e_t\]
\[+e_xe_t   \ot xte_{y^t}e_{z^t}
+e_y e_z\ot yze_{x^{yz}}e_t
-e_y  e_t\ot yte_{x^{yt}}e_{z^t}+\cdots
\]
so finally 
$(f\smile g)(e_xe_ye_ze_t)$
is equal to
\[
f(e_xe_y)g(e_ze_t)+f(e_ze_t)g( e_{x^{zt}}e_{y^{zt}})
-f(e_x e_z)g(e_{y^z}e_t)\]
\[
+f(e_xe_t)g(e_{y^t}e_{z^t})
+f(e_y e_z)g(e_{x^{yz}}e_t)
-f(e_y  e_t)g(e_{x^{yt}}e_{z^t}).
\]
This formula is to be compared with Equation (23) of \cite{Cl}. A full explanation of this agreement is given in the next section.
\end{example}

The last piece of structure to be extracted from Proposition~\ref{Bvscomplejo2} is the $A$-action:
\begin{prop}\label{prop:ActionOnRackCohom}
For a shelf $X$, the complex $(C^\bullet(X),\partial^*)$ is a left $A(X)$-module, with
\[(x\cdot f )( x_1 \cdots x_n ) := (x_1^x \cdots x_n^x),\]
where $f \in C^n(X)$, and $x,x_1, \ldots,x_n \in X$. The induced $A(X)$-action in cohomology is trivial.
\end{prop}

\begin{proof}
This  directly follows from Proposition~\ref{prop:dgb_hom} and~\ref{Bvscomplejo2}.
\end{proof}

This property of rack cohomology was first noticed by J.~Przytycki and K.~Putyra \cite{PrzPuLattices}. In our bialgebraic interpretation it becomes particularly natural.

\section{An explicit expression for the cup product in cohomology}

To give an explicit formula for the cup product in rack cohomology, we need to compute $\Delta(e_{x_1}\cdots e_{x_n})$ for any $x_1,\ldots,x_n$ in the rack $X$, generalizing Example~\ref{exDelta}. For this we will introduce several notations. First, for any $n\ge 1$ and for any $i\in\{1,\ldots, n\}$ we define two maps $\delta_i^0,\delta_i^1 \colon X^n \to X^{n-1}$ by
\begin{eqnarray*}
\delta_i^0(x_1,\ldots,x_n)&=&(x_1,\ldots,x_{i-1},x_{i+1}, \ldots ,x_n),\\
\delta_i^1(x_1,\ldots,x_n)&=&(x_1\t x_i,\ldots,x_{i-1}\t x_i,x_{i+1},\ldots,x_n).
\end{eqnarray*}
The above identification of $B$ with $A\otimes \ZZ\langle X\rangle$ given by $ae_{x_1}\cdots e_{x_n}\leftrightarrow a\ot x_1\cdots x_n$ allows one to transport $\delta_i^0$ and $\delta_i^1$ to $A$-linear endomorphisms of $B$:
\begin{eqnarray*}
\delta_i^0(ae_{x_1}\cdots e_{x_n})&=&ae_{x_1}\cdots e_{x_{i-1}}e_{x_{i+1}}\cdots e_{x_n}\hbox{ if }i\le n,\\
	&=& 0\hbox{ if }i>n,\\
\delta_i^1(ae_{x_1}\cdots e_{x_n})&=&ax_ie_{x_1\t x_i}\dots e_{x_{i-1}\t x_i}e_{x_{i+1}}\cdots e_{x_n}\hbox{ if }i\le n,\\
	&=& 0\hbox{ if }i>n.
\end{eqnarray*}
A straightforward computation using self-distributivity yields
\begin{equation}\label{cube-sets}
\delta_i^\varepsilon\delta_{j}^\eta=\delta_{j-1}^\eta\delta_i^\varepsilon
\end{equation}
for any $i<j$ and any $\varepsilon,\eta\in\{0,1\}$. Identities \eqref{cube-sets} are the defining axioms for $\square$-sets (\cite{SerreThesis}, see also \cite{FRS95,Cl}). Now, the boundary \eqref{bord-racks} can be rewritten as
\begin{equation}
\partial=\sum_{i\ge 1} (-1)^{i-1}(\delta_i^0-\delta_i^1).
\end{equation}
For any finite subset $S$ of $\mathbb N$ 
 and for $\varepsilon\in\{0,1\}$, we denote by $\delta_S^\varepsilon$ the composition, in the increasing order, of the maps $\delta_a^\varepsilon$ for $a\in S$.
 
\begin{prop}\label{coprod-explicit}
Given a rack~$X$, the coproduct in~$B(X)$ can be computed by the formula
\begin{equation}\label{coprod-formula}
\Delta(ae_{x_1}\cdots e_{x_n})=\sum_{S\subset\{1,\ldots,n\}}\epsilon(S)a\delta_S^0(e_{x_1}\cdots e_{x_n})\ot  a\delta_{S^c}^1(e_{x_1}\cdots e_{x_n})
\end{equation}
for all $a \in \langle X\rangle, x_i \in X$. Here $S^c=\{1,\ldots,n\} \setminus S$, 
 and $\epsilon(S)$ is the signature of the unshuffle permutation of $\{1,\ldots,n\}$ which puts  $S^c$ on the left and $S$ on the right. 
\end{prop}

We used the canonical form $ae_{x_1}\cdots e_{x_n}$ of a monomial in $B(X)$.

\begin{proof}
Since $\Delta(a)=a\ot a$, we can omit this part of our monomial. Let us then proceed by induction on $n$, the case $n=1$ being immediate.
\begin{align*}
\Delta(e_{x_1}&\cdots e_{x_n})=\Delta(e_{x_1}\cdots e_{x_{n-1}})\Delta(e_{x_n})\\
&=\left(\sum_{B\subset\{1,\ldots,n-1\}}\epsilon(B)\delta_B^0(e_{x_1}\cdots e_{x_{n-1}})\otimes \delta_{B^c}^1(e_{x_1}\cdots e_{x_{n-1}})
\right)(e_{x_n}\otimes x_n+1\otimes e_{x_n})\\
&=\sum_{B\subset\{1,\ldots,n-1\}}(-1)^{\vert B\vert}\epsilon(B)\delta_B^0(e_{x_1}\cdots e_{x_{n-1}})e_{x_n}\otimes \delta_{B^c}^1(e_{x_1}\cdots e_{x_{n-1}})x_n\\
&\qquad+\sum_{B\subset\{1,\ldots,n-1\}}\epsilon(B)\delta_B^0(e_{x_1}\cdots e_{x_{n-1}})\otimes \delta_{B^c}^1(e_{x_1}\cdots e_{x_{n-1}})e_{x_n}\\
&=\sum_{S\subset\{1,\ldots,n\},\ n \notin S}\epsilon(S)\delta_S^0(e_{x_1}\cdots e_{x_n})\otimes \delta_{S^c}^1(e_{x_1}\cdots e_{x_n})\\
&\qquad+\sum_{S\subset\{1,\ldots,n\},\ n \in S}\epsilon(S)\delta_S^0(e_{x_1}\cdots e_{x_n})\otimes \delta_{S^c}^1(e_{x_1}\cdots e_{x_n})\\
&=\sum_{S\subset\{1,\ldots,n\}}\epsilon(S)\delta_S^0(e_{x_1}\cdots e_{x_n})\otimes \delta_{S^c}^1(e_{x_1}\cdots e_{x_n}).\qedhere
\end{align*}
\end{proof}

\begin{cor}\label{clauwens}
The cup product in rack cohomology induced from the coproduct in $B$ coincides with the cup product given by F.J.-B.J. Clauwens in \cite[Equation (32)]{Cl}.
\end{cor}
\begin{proof}
This is immediate by comparing~\eqref{coprod-formula} with Equation~(32) defining the cup product in~\cite{Cl}. The overall sign $(-1)^{km}$ in~\cite{Cl} is the Koszul sign.
\end{proof}

\section{The cup product is commutative}\label{S:h}

In the preceding section we established that our cup product on rack cohomology coincides with Clauwens's product. This latter comes from the cohomology of a topological space, and is thus commutative (where, as usual, we mean super-commutativity). We will now give a direct algebraic proof based on an explicit homotopy argument. This homotopy is a specialization of the graphically defined map, constructed for solutions to the Yang--Baxter equation by V.~Lebed \cite{L16}.

\medskip
Let us start with a low-degree illustration:

\begin{example}
Take $f,g\in C^1(X)$, identified with $A$-linear maps from $B$ to $\ZZ$ (also denoted by $f$ and $g$) determined by the values $f(e_x):=f(x)$ and $g(e_x):=g(x)$ for $x\in X$, and vanishing in degrees other than $1$. Then the cup product $f\smile g\in C^2(X)$ is defined by
\[(f\smile g)(e_xe_y)=
(f\ot g)\Delta(e_xe_y)
=(f\ot g)(
e_xe_y\ot xy+1\ot e_xe_y+
e_x\ot xe_y
-e_y\ot ye_{x^y})
\]
(see Example~\ref{exDelta}). Since $f$ and $g$ vanish on elements of degree $0$ and $2$, and are left $A$-linear (where $x$ and $y$ act on $\ZZ$ trivially), the only remaining terms are
\[
-f(e_x)g( xe_y)
+f(e_y)g( ye_{x^y})=
-f(e_x)g( e_y)
+f(e_y)g(e_{x^y}).
\]
Note the Koszul sign $(-1)^{|g||e_x|}=-1$, and similarly in the second term. So the product is in general not commutative. 
 On the other hand, the cocycle condition $\partial^*g=0$ means precisely $g(e_x)=g(e_{x^y})$ for all $x$ and $y$, so the cup product restricted to $1$-cocycles is commutative.

Now, take $f,g\in C^1(X,M(X))$, identified with maps $B \to \ZZ$ vanishing in degrees other than $1$. Then, for for monomials $a\in A$ and $x,y \in X$, one computes 
\begin{align}\label{E:cup}
&(f\smile g)(ae_xe_y)=-f(ae_x)g(axe_y)+f(ae_y)g(ae_{x}y),\nonumber\\
&(f\smile g+g\smile f)(ae_xe_y) =\nonumber\\
&\qquad -f(ae_x)g(axe_y)+f(ae_y)g(ae_{x}y)-g(ae_x)f(axe_y)+g(ae_y)f(ae_{x}y).
\end{align}
Suppose that $f$ and $g$ are $1$-cocycles. This yields the relation
\begin{align*}
0&=(-d^*f)(ae_xe_y)=f(d(ae_xe_y))=f(d(a)e_xe_y)+f(ad(e_x)e_y)-f(ae_xd(e_y))\\
&=f(a(1-x)e_y)-f(ae_x(1-y))=f(ae_y)-f(axe_y)-f(ae_x)+f(ae_xy),
\end{align*}
and similarly for $g$. Note the Koszul sign in $-d^*f=fd$. Define a map $h\colon B \to \ZZ$ by
\[h(ae_x)=f(ae_x)g(ae_x)\]
for all monomials $a\in A$ and all $x \in X$. Then \eqref{E:cup} becomes
\begin{align*}
&-f(ae_x)g(axe_y)+f(ae_y)(g(ae_x)-g(ae_y)+g(axe_y))\\
&\qquad-g(ae_x)(f(ae_xy)-f(ae_x)+f(ae_y))+g(ae_y)f(ae_{x}y)\\
&=(f(ae_y)-f(ae_x))g(axe_y)-h(ae_y)+h(ae_x)+(g(ae_y)-g(ae_x))f(ae_{x}y)\\
&=(f(axe_y)-f(ae_xy))g(axe_y)-h(ae_y)+h(ae_x)+(g(axe_y)-g(ae_xy))f(ae_{x}y)\\
&=h(axe_y)-h(ae_y)+h(ae_x)-h(ae_{x}y)=(d^*h)(ae_xe_y),
\end{align*}
yielding the relation $f\smile g+g\smile f=d^*h$, and hence the super-commutativity of the cup product of degree $1$ cohomology classes.

\end{example}

\begin{lem}\label{com-one}
Let $h\colon B\to B\ot B$ be the degree $1$ 
linear map defined on monomials in~$B$ written in the canonical form as follows: $h(a)=0$, and
\begin{eqnarray*}
h(ae_{x_1}\cdots e_{x_n})&:=&\sum_{i=1}^n (-1)^{i-1}(a\ot a) (\tau\Delta)(e_{x_1}\cdots e_{x_{i-1}})(e_{x_i} \ot e_{x_i})\Delta(e_{x_{i+1}}\cdots e_{x_n}),
\end{eqnarray*}
where $\tau\colon B\ot B\to B\ot B$ is the signed flip.  
 Then for any homogeneous $b_1,b_2\in B$ we have
\begin{equation}\label{eq:com-one}
h(b_1b_2)=h(b_1)\Delta(b_2)+(-1)^{|b_1|}(\tau\Delta)(b_1)h(b_2).
\end{equation}
Also, $h$ induces a map $\oB \to \oB \ot \oB$.
\end{lem}

The induced map will still be denoted by~$h$.

\begin{proof}
Using the fact that both $\Delta$ and $\tau\Delta$ are algebra morphisms, one rewrites the definition of~$h$ as
\begin{align*}
h(ae_{x_1}\cdots e_{x_n})=&\sum_{i=1}^n (-1)^{i-1}(a\ot a)(\tau\Delta)(e_{x_1})\cdots (\tau\Delta)(e_{x_{i-1}})(e_{x_i} \ot e_{x_i})\Delta(e_{x_{i+1}})\cdots\Delta(e_{x_n}).
\end{align*}
This immediately yields \eqref{eq:com-one} on $b_1=a_1e_{x_1}\cdots e_{x_p}$ (i.e., any monomial in $B$) and $b_2=e_{x_{p+1}}\cdots e_{x_{p+q}}$. To check \eqref{eq:com-one} on general monomials $b_1$ and $b'_2=a_2b_2$, with $a_2 \in \langle X \rangle$ and $b_2$ a product of the $e_x$'s, one observes that the maps $h$, $\Delta$, and $\tau\Delta$ are $X$-equivariant both on the left and on the right, which gives
\begin{align*}
h(b_1(a_2b_2))&=h((b_1a_2)b_2) = h(b_1a_2)\Delta(b_2)+(-1)^{|b_1a_2|}(\tau\Delta)(b_1a_2)h(b_2) \\
&= h(b_1)(a_2\ot a_2)\Delta(b_2)+(-1)^{|b_1|}(\tau\Delta)(b_1)(a_2\ot a_2)h(b_2) \\
&= h(b_1)\Delta(a_2b_2)+(-1)^{|b_1|}(\tau\Delta)(b_1)h(a_2b_2).
\end{align*}

Finally, $h$ survives in the quotient $\oB$ since its $X$-equivariance reads $h(ae_{x_1}\cdots e_{x_n}) = (a \ot a) h(e_{x_1}\cdots e_{x_n}) \sim h(e_{x_1}\cdots e_{x_n})$, with the usual notations.
\end{proof}

For example, an easy computation gives
\[h(e_xe_y)=(xe_y+e_x)\otimes e_xe_y-e_xe_y\otimes(e_xy+e_y),\]
which in $\oB$ becomes
\[h(e_xe_y)=(e_y+e_x)\otimes e_xe_y-e_xe_y\otimes(e_{x^y}+e_y).\]

\begin{prop}\label{homotopy}
The map $h$ is a homotopy between $\Delta$ and $\tau\Delta$: 
\begin{equation}\label{eq:homotopy}
(d\ot \Id_B+\Id_B\ot d) h+h d=\Delta-\tau\Delta.
\end{equation}
\end{prop}

Of course, $h$ remains a homotopy in $\oB$ as well.

\begin{proof}
We will use the short-hand notation $dh := (d\ot \Id_B+\Id_B\ot d) h$.

If $x$ is a degree zero generator of $B$ we have $(dh+hd)(x)=0$, and $\Delta(x)=x\otimes x=(\tau\Delta)(x)$, hence \eqref{eq:homotopy} holds. Now, for a degree one generator $e_x$ we have
\begin{eqnarray*}
(dh+hd)(e_x)&=&d(e_x\otimes e_x)+0\\
&=&(1-x)\otimes e_x-e_x\otimes (1-x)\\
&=&-x\otimes e_x+e_x\otimes x+1\otimes e_x-e_x\otimes 1\\
&=&(\Delta-\tau\Delta)(e_x).
\end{eqnarray*}
The proof can then be carried out by induction on the degree, using \eqref{eq:com-one}:
\begin{eqnarray*}
(dh+hd)(b_1b_2)&=&d\big(h(b_1)\Delta(b_2)+(-1)^{|b_1|}(\tau \Delta)(b_1)h(b_2)\big)+h\big(db_1 \cdot b_2+(-1)^{|b_1|}b_1\cdot db_2\big)\\
&=&dh(b_1)\Delta(b_2)+(-1)^{|b_1|+1}h(b_1)d\Delta(b_2)\\
&&+(-1)^{|b_1|}d(\tau\Delta)(b_1)h(b_2)+(\tau\Delta)(b_1)dh(b_2)\\
&&+hd(b_1)\Delta(b_2)+(-1)^{|b_1|+1}(\tau\Delta)(db_1)h(b_2)\\
&&+(-1)^{|b_1|}h(b_1)\Delta(db_2)+(\tau\Delta)(b_1)hd(b_2)\\
&=&(\Delta-\tau\Delta)(b_1)\Delta(b_2)+(\tau\Delta)(b_1)(\Delta-\tau\Delta)(b_2)\\
&=&\Delta(b_1)\Delta(b_2) - (\tau\Delta)(b_1)(\tau\Delta)(b_2)\\
&=&(\Delta-\tau\Delta)(b_1b_2). \qedhere
\end{eqnarray*}
\end{proof}

\begin{thm}\label{cup-product}
For a shelf $X$, the map $h$ induces a homotopy between the cup product $\smile$ and its opposite version $\smile^{op}:=\smile \tau$ on $C^\bullet(X)$ and $C^\bullet(X,M(X))$. 
\end{thm}

Using a standard argument, we obtain an elementary algebraic proof of the commutativity of the cup product on the rack cohomology $H_{\Rack}(X)$ and $H_{\Rack}(X,M(X))$. The same result holds for the more general cohomologies  $H_{\Rack}(X,k)$ and $H_{\Rack}(X,kM(X))$. 

\begin{proof}
The cup product of two cochains $f$ and $g$ is given by the convolution product
\[f\smile g=\mu (f\ot g)\Delta,\]
where the coproduct $\Delta$ is taken in $\oB$, and $\mu$ is the multiplication in $\ZZ$. Hence for any homogeneous $x\in \oB$ of degree $|f|+|g|$ we have
\begin{eqnarray*}
\big(f\smile g-(-1)^{|f||g|}g\smile f\big)(x)&=&\sum_{(x)}(-1)^{|f||g|}f(x_1)g(x_2)-g(x_1)f(x_2)\\
&=&\sum_{(x)}(-1)^{|x_1||x_2|}f(x_1)g(x_2)-f(x_2)g(x_1)\\
&=&\mu (f\otimes g) (\Delta-\tau\Delta)(x)\\
&=&\mu (f\otimes g) (hd+dh)(x).
\end{eqnarray*}
We used Sweedler's notation for $\Delta(x)$. Hence $H\colon \mop{Hom}_A(B,\ZZ)^{\ot 2}\to\mop{Hom}_A(B,\ZZ)$ defined by
\[H(f\ot g):=\mu (f\otimes g) h\]
is a homotopy between $\smile$ and $\smile^{op}$. The proof for the cohomology with coefficients in $M$ is similar.
\end{proof}

\section{Rack cohomology is a Zinbiel algebra}\label{S:Z}

We now refine the coproduct~$\Delta$ on $\oB(X)$ to an (almost) \emph{d.g. codendriform structure}. That is, in positive degree it decomposes as $\Delta=\lDelta+\rDelta$, the two parts $\lDelta$ and $\rDelta$ being compatible. Moreover, we establish the relation $\rDelta = \tau\lDelta$ (where $\tau$ is as usual the signed flip), up to an explicit homotopy~$\oh$. This latter is inspired by the homotopy~$h$ from Section~\ref{S:h}, and is, to our knowledge, new.  We thus recover the \emph{Zinbiel product} on rack cohomology, first described by S.~Covez in~\cite{C2}. 

Coalgebras need not be unital in this section. General definitions are given over a unital commutative ring $k$; in particular, all the tensor products are taken over~$k$ here.

\begin{defn}
A graded coalgebra $(C=\oplus_{i \geqslant 0}C_i,\Delta)$ is called \emph{$+$-codendriform} if there exist two maps $\lDelta \colon C^+ \to C^+ \ot C$ and $\rDelta \colon C^+\to C \ot C^+$ of degree $0$ satisfying
\begin{align}
(\lDelta\ot\Id)\lDelta=(\Id\ot\Delta)\lDelta,\label{E:dendri1}\\
(\Id\ot\rDelta)\rDelta=(\Delta\ot\Id)\rDelta,\label{E:dendri2}\\
(\Id\ot\lDelta)\rDelta=(\rDelta\ot\Id)\lDelta,\label{E:dendri3}
\end{align}
$\Delta$ decomposes as $\lDelta+\rDelta$ on $C^+$, and $\Delta$ is coassociative on $C_0$. Here $C^+=\oplus_{i \geqslant 1}C_i$. It is called \emph{$+$-coZinbiel} if moreover $\rDelta = \tau\lDelta$, where $\tau$ is the signed flip. A \emph{d.g. $+$-codendriform / $+$-coZinbiel coalgebra} carries in addition a differential $d$ satisfying
\begin{align}
\lDelta d &=(d \ot \Id)\lDelta + (\Id\ot d)\lDelta \qquad \text{ on } \oplus_{i \geqslant 2}C_i,\label{E:dendri4}\\
\rDelta d &=(d \ot \Id)\rDelta + (\Id\ot d)\rDelta \qquad \text{ on } \oplus_{i \geqslant 2}C_i,\label{E:dendri5}\\
\Delta d &=(d \ot \Id)\lDelta + (\Id\ot d)\rDelta \qquad \text{ on } C_1.\label{E:dendri6}
\end{align}
Dually, one defines (d.g.) \emph{$+$-dendriform} and \emph{$+$-Zinbiel} algebras.
\end{defn}

A typical example of a $+$-codendriform coalgebra is a positively graded codendriform coalgebra $(C^+,\lDelta^+,\rDelta^+)$, extended by a unit: $C:=C^+\oplus k1$, with $\Delta(1)=\lDelta(1)=\rDelta(1)=1\ot 1$, and $\lDelta(c)=\lDelta^+(c)+c\ot 1$ and $\rDelta(c)=\rDelta^+(c)+1 \ot c$ for all $c \in C$. One can also go in the opposite direction:

\begin{lem}\label{L:dedndri}
Let $(C,\Delta,\lDelta,\rDelta)$ be a $+$-codendriform coalgebra. Denote by $\varepsilon \colon C \to C^+$ and $\iota \colon C^+ \to C$ the obvious projection and inclusion, where $C^+:= \oplus_{i > 0}C_i$. Put $\Delta^+:=(\varepsilon \ot \varepsilon) \Delta \iota$, $\lDelta^+:=(\varepsilon \ot \varepsilon) \lDelta \iota$, $\rDelta^+:=(\varepsilon \ot \varepsilon) \rDelta \iota$. Then $(C^+,\lDelta^+,\rDelta^+)$ is a codendriform coalgebra.
\end{lem}

The proof is straightforward. These observations explain our choice of the name. In the literature there exist alternative approaches to such ``almost codendriform'' structures.

Finally, one easily checks that a $+$-codendriform structure refines a coassociative one: 

\begin{lem}\label{L:dedndri_ass}
In a (d.g.) $+$-codendriform coalgebra, the coproduct~$\Delta$ is necessarily coassociative (and compatible with the differential).
\end{lem}

Let us now return to shelves and their associated d.g. bialgebras.
\begin{prop}\label{P:dendri}
Let $X$ be a shelf. Define two maps $\lDelta \colon B(X)^+\to B(X)^+ \ot B(X)$ and $\rDelta \colon B(X)^+ \to B(X) \ot B(X)^+$ as follows:
\begin{align*}
&\lDelta(ae_{x_1}\cdots e_{x_n})=(a e_{x_1} \ot a x_1) \Delta (e_{x_2}\cdots e_{x_n}),\\ 
&\rDelta(ae_{x_1}\cdots e_{x_n})=(a \ot a e_{x_1}) \Delta (e_{x_2}\cdots e_{x_n}),
\end{align*}
where as usual we use the canonical form of monomials in $B(X)$, and extend this definition by linearity. These maps and the coproduct $\Delta$ yield a $+$-codendriform structure on~$B(X)$.
\end{prop}

\begin{proof}
Put $\Delta^2= (\Delta\ot\Id)\Delta=(\Id\ot\Delta)\Delta$. Then both sides of~\eqref{E:dendri1} act on a canonical monomial as follows:
\[ae_{x_1}\cdots e_{x_n} \mapsto (a e_{x_1} \ot a x_1 \ot a x_1)\Delta^2 (e_{x_2}\cdots e_{x_n}).\]
Similarly, both sides of~\eqref{E:dendri2} and~\eqref{E:dendri3} act by 
\begin{align*}
ae_{x_1}\cdots e_{x_n} &\mapsto (a \ot a \ot a e_{x_1})\Delta^2 (e_{x_2}\cdots e_{x_n}) \quad \text{ and}\\
ae_{x_1}\cdots e_{x_n} &\mapsto (a \ot a e_{x_1}  \ot a x_1)\Delta^2 (e_{x_2}\cdots e_{x_n})
\end{align*}
respectively. Thus our maps satisfy relations \eqref{E:dendri1}-\eqref{E:dendri3}. Finally, in positive degree their sum clearly yields $\Delta$, and in degree $0$ the coproduct $\Delta$ is coassociative.
\end{proof}

The maps above are not compatible with the differential in general: one has
\begin{align*}
\lDelta d (e_xe_y) &= e_y \ot y - xe_y \ot xy - e_x \ot x + e_x y \ot xy,\\
(d \ot \Id)\lDelta + (\Id\ot d)\lDelta (e_xe_y) &= e_y \ot xy - xe_y \ot xy - e_x \ot x + e_x y \ot xy\\
&+ 1 \ot xe_y - x \ot xe_y.
\end{align*}

As usual, the solution is to work in the quotient~$\oB(X)$. Indeed, $\lDelta$ and $\rDelta$ descend to maps $\oB(X)^+ \to \oB(X)^+ \ot \oB(X)$ and $\oB(X)^+ \to \oB(X) \ot \oB(X)^+$, still denoted by $\lDelta$ and $\rDelta$, and one has:

\begin{prop}\label{P:dendriH}
The induced maps $\lDelta$ and $\rDelta$ make $\oB(X)$ a d.g. $+$-codendriform coalgebra.
\end{prop}

\begin{proof}
Recall the interpretation \eqref{E:BBar} of $\oB$ as the quotient of $B$ by $xb \sim b$ for all $x \in X,\ b \in B$. It yield the maps $\lDelta$ and $\rDelta$ symmetric in~$\oB$:
\begin{align*}
&\lDelta(e_{x_1}\cdots e_{x_n})=(e_{x_1} \ot 1) \Delta (e_{x_2}\cdots e_{x_n}), \qquad \rDelta(e_{x_1}\cdots e_{x_n})=(1 \ot e_{x_1}) \Delta (e_{x_2}\cdots e_{x_n}).
\end{align*}
Also, it turns \eqref{E:L1property} into
\begin{align}\label{E:dBbar}
d (e_{x_1}\cdots e_{x_n}) &=-e_{x_1}d (e_{x_2}\cdots e_{x_n}).
\end{align}
We can now establish relation~\eqref{E:dendri4}:
\begin{align*}
\lDelta d (e_{x_1}\cdots e_{x_n}) &=\lDelta(-e_{x_1}d (e_{x_2}\cdots e_{x_n}))\\
&=-(e_{x_1} \ot 1) \Delta d (e_{x_2}\cdots e_{x_n})\\
&=-(e_{x_1} \ot 1)  (d \ot \Id) \Delta (e_{x_2}\cdots e_{x_n})-(e_{x_1} \ot 1)  (\Id \ot d) \Delta (e_{x_2}\cdots e_{x_n})\\
&=(d \ot \Id)(e_{x_1} \ot 1) \Delta (e_{x_2}\cdots e_{x_n})+(\Id \ot d) (e_{x_1} \ot 1)   \Delta (e_{x_2}\cdots e_{x_n})\\
&=(d \ot \Id + \Id \ot d) \lDelta (e_{x_1}\cdots e_{x_n}).
\end{align*}
Relation~\eqref{E:dendri5} is proved similarly. Finally, relation~\eqref{E:dendri6} follows from $\Delta d = (d \ot \Id + \Id \ot d) \Delta$ in degree~$1$.
\end{proof}

\begin{prop}\label{P:Zinbiel}
Define the map $\oh \colon B(X)\to B(X) \ot B(X)$ as follows: $\oh(a)=0$, and
\begin{align*}
\oh(ae_{x_1}\cdots e_{x_n})&=-(a {x_1} \ot a e_{x_1}) h(e_{x_2}\cdots e_{x_n}).
\end{align*}
It induces a map $\oB \to \oB \ot \oB$, still denoted by~$\oh$, which is a homotopy between $\rDelta$ and $\tau\lDelta$.
\end{prop}

\begin{proof}
The map $\oh$ clearly descends to $\oB$. For this induced map, one has
\begin{align*}
\oh(e_{x_1}\cdots e_{x_n})&=-(1\ot e_{x_1}) h(e_{x_2}\cdots e_{x_n}).
\end{align*}

It remains to check the relation
\[(d\ot \Id_{\oB}+\Id_{\oB}\ot d) \oh+ \oh d=\tau\lDelta-\rDelta \colon \oB^+ \to \oB \ot \oB.\]
Using~\eqref{E:dBbar}, one computes
\begin{align*}
(d\ot \Id) \oh (e_{x_1}\cdots e_{x_n}) &=-(d\ot \Id)(1 \ot e_{x_1}) h(e_{x_2}\cdots e_{x_n}) \\
&=(1 \ot e_{x_1})(d\ot \Id) h(e_{x_2}\cdots e_{x_n}),\\
(\Id\ot d) \oh (e_{x_1}\cdots e_{x_n}) 
&=-(\Id\ot d)(1\ot e_{x_1}) h(e_{x_2}\cdots e_{x_n}) \\
&=(1 \ot e_{x_1})(\Id\ot d) h(e_{x_2}\cdots e_{x_n}),\\
\oh d (e_{x_1}\cdots e_{x_n})&= - \oh(e_{x_1} d(e_{x_2}\cdots e_{x_n})) \\
&=(1 \ot e_{x_1}) hd(e_{x_2}\cdots e_{x_n}).
\end{align*}
The sum yields
\begin{align*}
(1 \ot e_{x_1})&\left((d\ot \Id+\Id\ot d) h + hd\right)(e_{x_2}\cdots e_{x_n})=(1 \ot e_{x_1})\left(\Delta- \tau\Delta \right)(e_{x_2}\cdots e_{x_n})\\
&=(1 \ot e_{x_1})\Delta (e_{x_2}\cdots e_{x_n})-\tau\left((e_{x_1} \ot 1)\Delta(e_{x_2}\cdots e_{x_n})\right)\\
&=\rDelta(e_{x_1}e_{x_2}\cdots e_{x_n})-\tau\lDelta(e_{x_1}e_{x_2}\cdots e_{x_n}),
\end{align*}
as desired. 
\end{proof}

As usual, using Lemma~\ref{Bvscomplejo} one deduces from Proposition~\ref{P:dendri} a $+$-dendriform structure on the complex defining rack cohomology, and from Proposition~\ref{P:Zinbiel} a $+$-Zinbiel product on the rack cohomology. Lemma~\ref{L:dedndri} then yields dendriform and Zinbiel structures in positive degree:

\begin{thm}\label{thm:Zinbiel}
For a shelf $X$, the complex $(\oplus_{n\geqslant 1}C^n(X),\partial^*)$ admits a dendriform structure, which is Zinbiel up to a homotopy induced by~$\oh$. The rack cohomology of~$X$ thus receives a strictly Zinbiel product. 
\end{thm}

\begin{rem}
The dendriform structures above are not surprising: in~\cite{Lebed1,L16}, rack cohomology is interpreted in terms of quantum shuffle algebras, which are key examples of dendriform structures. The shuffle interpretation generalizes to the cohomology of solutions to the Yang--Baxter equation, where dendriform structures reappear as well. The Zinbiel structure in cohomology is on the contrary remarkable, and does not extend to the YBE setting. 
Shuffles also suggest that for $B(X)^+$, the codendriform structure and the associative product are compatible, in the sense of~\cite{Ronco}.  However, this does not seem to yield Zinbiel-coassociative structures on rack cohomologies: if we choose to work without coefficients (i.e. in $\oB(X)^+$), the dendriform structure is compatible with the differential, but the coproduct is lost; if we take coefficients $kM(X)$ (i.e. we work in $B(X)^+$), where $k$ is a field and $X$ is finite, the coproduct is preserved, but the dendriform structure does not survive in cohomology. 
 
\end{rem}

\section{Quandle cohomology inside rack cohomology}\label{S:Q}

If $X$ is a spindle (e.g., a quandle), then the complex $C_\bullet(X,k)$ has a \emph{degenerate subcomplex}
\[
 C_\bullet^{\Deg}(X,k)=\langle \ x^2\ :\ x \in X\ \rangle.
\]
In other words, it is the linear envelope of all monomials with repeating neighbours. J.S. Carter et al. \cite{QuandleHom} defined the \emph{quandle (co)homology} of $X$ via the complexes
\[
 C_\bullet^{\Quandle}(X,k):=C_\bullet(X,k)/C_\bullet^{\Deg}(X,k)\hbox{, and }
 C^\bullet_{\Quandle}(X,k):=\Hom(C^{\Quandle}_\bullet(X,k),k).
\]
R.A. Litherland and S.~Nelson \cite{LiNe} showed that in this case the complex $C_\bullet(X,k)$ splits:
\[C = C^{\Norm} \oplus C^{\Deg}.\] 
The quandle (co)homology is then the (co)homology of the complement $C^{\Norm}$. We will now show that this decomposition is already visible at the level of the d.g. bialgebra $B(X)$. Moreover, in the bialgebraic setting it will be particularly easy to prove that:
\begin{itemize}
\item the Zinbiel product on rack cohomology induces one on quandle cohomology, but does not restrict to quandle cohomology;
\item the associative cup product on rack cohomology does restrict to quandle cohomology.
\end{itemize}

\begin{prop}\label{P:Q1}
Let $X$ be a spindle. In $B(X)$, consider the ideal
\[B^{\Deg}(X):=\langle \ e_x^2 \ :\ x\in X\ \rangle,\]
and the left sub-$A(X)$-module $B^{\Norm}(X)$ generated by the elements $1$ and
\begin{equation}\label{E:complement}
(e_{x_1}-e_{x_2})(e_{x_2}-e_{x_3})\cdots(e_{x_{n-1}}-e_{x_n})e_{x_n}, \qquad \text{ where } n \geqslant1, \text{ and all } x_i \in X.
\end{equation}
Then $B$ decomposes as a d.g. $A$-bimodule:
\begin{equation}\label{E:decomposition}
B(X) = B^{\Norm}(X) \oplus B^{\Deg}(X).
\end{equation}
Moreover, $B^{\Deg}$ is a coideal, and $B^{\Norm}$ is a left coideal and a left codendriform coideal of $B$.
\end{prop}

\begin{proof} 
The expression~\eqref{E:complement} vanishes when $x_i=x_{i+1}$ for some $i$. Moreover, one has
\begin{align*}
e_{x_1}\cdots e_{x_n} = (e_{x_1}-e_{x_2})&(e_{x_2}-e_{x_3})\cdots(e_{x_{n-1}}-e_{x_n})e_{x_n} \ + \ \text{terms from } B^{\Deg}.
\end{align*}
This implies the decomposition~\eqref{E:decomposition} of abelian groups. 

The subspaces $B^{\Norm}$ and $B^{\Deg}$ are homogeneous, and for any $y \in X$ one has 
\[(e_{x_1}-e_{x_2})\cdots(e_{x_{n-1}}-e_{x_n})e_{x_n} y = 
y (e_{x_1^y}-e_{x_2^y})\cdots(e_{x_{n-1}^y}-e_{x_n^y})e_{x_n^y}.\]
So, $B^{\Norm}$ and $B^{\Deg}$ are graded sub-$A$-bimodules of~$B$. 

Let us now check that $B^{\Deg}$ is a differential coideal. Indeed, using the property $x^x=x$ of a spindle, one computes
\begin{align}
d(e_x^2)&=d(e_x)e_x-e_xd(e_x)=(1-x)e_x-e_x(1-x)=e_x-xe_x-e_x+xe_{x^x}=0,\label{E:dBD}\\
\Delta(e_x^2)
&=e_x^2\ot x^2+1\ot e_x^2+
e_x\ot xe_x
-e_x\ot xe_{x^x}
=e_x^2\ot x^2+1\ot e_x^2.\label{E:deltaBD}
\end{align}

To check that $B^{\Norm}$ is a subcomplex of~$B$, we need its alternative description:

\begin{lem}\label{L:Q}
 $B^{\Norm}(X)$ is the left sub-$A(X)$-module generated by the elements $1$ and
\begin{equation}\label{E:complement2}
(e_{x_1}-e_{y_1})(e_{x_2}-e_{y_2})\cdots(e_{x_{n-1}}-e_{y_{n-1}})e_{x_n}, \qquad \text{ where } n \geqslant1, \text{ and all } x_i, y_i \in X.
\end{equation}
\end{lem}

\begin{proof}
It is sufficient to represent an element of the form~\eqref{E:complement2} as a linear combination of elements of the form~\eqref{E:complement}. This can be done inductively using the following observation:
\begin{align*}
(e_x-e_y)&(e_{x_1}-e_{x_2})(e_{x_2}-e_{x_3})\cdots(e_{x_{n-1}}-e_{x_n})e_{x_n} =\\
&(e_x-e_{x_1})(e_{x_1}-e_{x_2})(e_{x_2}-e_{x_3})\cdots(e_{x_{n-1}}-e_{x_n})e_{x_n}\\
-&(e_y-e_{x_1})(e_{x_1}-e_{x_2})(e_{x_2}-e_{x_3})\cdots(e_{x_{n-1}}-e_{x_n})e_{x_n}. \qedhere
\end{align*}
\end{proof}  

Now, for $a \in A$ and $x_1,\ldots,x_n \in X$, we have
\begin{align*}
d(a(e_{x_1}&-e_{x_2})(e_{x_2}-e_{x_3})\cdots(e_{x_{n-1}}-e_{x_n})e_{x_n})\\
=&ad(e_{x_1}-e_{x_2})(e_{x_2}-e_{x_3})\cdots(e_{x_{n-1}}-e_{x_n})e_{x_n}\\
&-a(e_{x_1}-e_{x_2})d\left( (e_{x_2}-e_{x_3})\cdots(e_{x_{n-1}}-e_{x_n})e_{x_n} \right) \\
=&a(x_2-x_1)(e_{x_2}-e_{x_3})\cdots(e_{x_{n-1}}-e_{x_n})e_{x_n}\\
&-a(e_{x_1}-e_{x_2})d\left( (e_{x_2}-e_{x_3})\cdots(e_{x_{n-1}}-e_{x_n})e_{x_n} \right).
\end{align*}
An inductive argument using the lemma shows that this lies in $B^{\Norm}$.

It remains to prove that $\Delta$, $\lDelta$ and $\rDelta$ send $B^{\Norm}$ to $B \otimes B^{\Norm}$. In degree $0$ everything is clear. In higher degree, from
\[\Delta(e_x-e_y)=e_x\otimes x - e_y \otimes y + 1 \otimes (e_x-e_y),\]
one sees that any of $\Delta$, $\lDelta$ and $\rDelta$ sends an expression of the form \eqref{E:complement2} to a linear combination of tensor products, where on the right one has a product of terms of the form $z$, $e_x-e_y$, and possibly an $e_u$ at the end. By Lemma~\ref{L:Q}, all these right parts lie in $B^{\Norm}$.
\end{proof}

The proposition describes all the structure inherited from $B$ by $B^{\Deg}$ and $B^{\Norm}$. Indeed,
\begin{itemize}
\item $B^{\Deg}$ is not a sub-coalgebra, as follows from~\eqref{E:deltaBD};
\item $B^{\Deg}$ is not a coideal in the dendriform sense, since
\begin{equation}\label{E:DendriExample}
\lDelta(e_x^2)=e_x^2\ot x^2+e_x\ot xe_x;
\end{equation} 
\item $B^{\Norm}$ is not a subalgebra of $B$: for any $x \in X$, $e_x$ lies in $B^{\Norm}$, whereas $e_x^2 \in B^{\Deg}$;
\item $B^{\Norm}$ is not a sub-coalgebra either: one has
\begin{equation}\label{E:DeltaBN}
\Delta \left((e_x-e_y)e_y \right)= e_y^2 \ot (x-y)y + \text{ terms from } B^{\Norm} \ot B^{\Norm},
\end{equation}
and $e_y^2 \ot (x-y)y$ is a non-zero term from $B^{\Deg} \ot B^{\Norm}$ in general.
\end{itemize}
In particular, there is no natural way to define a codendriform structure on $B^{\Norm}$. Passing to the quotient $\oB$ does not solve this problem: $\oB^{\Deg}$ is still not a codendriform coideal because of \eqref{E:DendriExample}, and $\oB^{\Norm}$ is not a sub-coalgebra of $\oB$. Indeed, even if \eqref{E:DeltaBN} implies $\Delta \left((e_x-e_y)e_y \right) \in \oB^{\Norm} \ot \oB^{\Norm}$, things go wrong in degree~$3$:
\begin{align*}
\Delta \left((e_x-e_y)(e_y-e_z)e_z \right)= e_z^2 \ot (e_{X^Y}-e_{X^z}-e_{Y}+e_{Y^z})+ \text{ terms from } \oB^{\Norm} \ot \oB^{\Norm},
\end{align*}
where $X=x^z$, $Y=y^z$. One gets a term from $\oB^{\Deg} \ot \oB^{\Norm}$ which does not vanish in general. However, since $ e_{X^Y} = e_X=e_{X^z}$ and $e_{Y^z}=e_{Y}$ modulo the boundary, this terms disappears in homology. More generally: 
é
\begin{prop}\label{P:Q}
Let $X$ be a spindle. The homology $\oH(X)$ of $\oB(X)$ decomposes as a graded abelian group:
\begin{equation}\label{E:decomposition_hom}
\oH(X) = \oH^{\Norm}(X) \oplus \oH^{\Deg}(X).
\end{equation}
If $k$ is a field, then one obtains a decomposition
\begin{equation}\label{E:decomposition_hom_field}
\oH(X,k) = \oH^{\Norm}(X,k) \oplus \oH^{\Deg}(X,k),
\end{equation}
with $\oH^{\Deg}$ a coassociative coideal and $\oH^{\Norm}$ a coZinbiel (and hence coassociative) coalgebra.

Dually, the cohomology $\oH^{\bullet}(X)$ of $\oB(X)^*$ decomposes as 
\begin{equation}\label{E:decomposition_cohom}
\oH^{\bullet}(X) = \oH^{\bullet}_{\Norm}(X) \oplus \oH^{\bullet}_{\Deg}(X),
\end{equation}
with $\oH^{\bullet}_{\Deg}$ a Zinbiel (and hence associative) ideal, and $\oH^*_{\Norm}$ an associative subalgebra of $\oH^{\bullet}$. The same holds for $\oH^{\bullet}(X,k)$.
\end{prop}

\begin{proof}
Proposition~\ref{P:Q1} yields the desired decompositions, and, together with Propositions~\ref{prop:dgb_hom_bialg} and~\ref{P:dendriH}, shows that $\oH^{\Deg}$ is a coideal and $\oH^{\Norm}$ a left codendriform coideal. In particular, 
\begin{align*}
\lDelta((\oH^{\Norm})^+) &\subseteq (\oH^{\Norm})^+ \ot \oH^{\Norm} \oplus (\oH^{\Deg})^+ \ot \oH^{\Norm},\qquad\text{ and }\\
\rDelta((\oH^{\Norm})^+) &\subseteq \oH^{\Norm} \ot (\oH^{\Norm})^+ \oplus \oH^{\Deg} \ot (\oH^{\Norm})^+.
\end{align*}
But Proposition~\ref{P:Zinbiel} yields the relation $\rDelta =\tau \lDelta$ in homology, hence the terms in $(\oH^{\Deg})^+ \ot \oH^{\Norm}$ and $\oH^{\Deg} \ot (\oH^{\Norm})^+$ above must be trivial. This shows that $\oH^{\Norm}$ is in fact a coZinbiel coalgebra. 

The proof for the cohomology $\oH^{\bullet}$ is analogous.
\end{proof}

Again, this proposition describes all the structure inherited by $\oH^{\Deg}(X,k)$: it is neither a sub-coalgebra, nor a codendriform coideal. Indeed, computations \eqref{E:deltaBD} and \eqref{E:DendriExample} still yield counterexamples, since $e_x^2$ and $e_x$ represent non-trivial classes in $\oH^{\Deg}$ and $\oH^{\Norm}$ respectively.

Now, in order to understand what our proposition means for quandle cohomology, we need to recall Lemma~\ref{Bvscomplejo} and observe that the construction of $B^{\Deg}$ precisely repeats that of the degenerate complex. This yields: 
\begin{lem}\label{BvscomplejoQ}
For any spindle $X$, one has isomorphisms of complexes
\[(C^{\Quandle}_\bullet,\partial)\cong \oB^{\Norm} \hbox{,\; and \; }(C^\bullet_{\Quandle},\partial^*)\cong \oB^*_{\Norm}.\]
\end{lem}

Proposition~\ref{P:Q} then translates as follows:

\begin{thm}\label{thm:Q}
The rack cohomology of a spindle $X$ decomposes into quandle and degenerate parts: one has the isomorphism
\[H_{\Rack}(X) \simeq H_{\Quandle}(X) \oplus H_{\Deg}(X)\]
of graded abelian groups. Moreover,
\begin{itemize}
\item $H_{\Quandle}$ is an associative subalgebra of $H_{\Rack}$, and $H_{\Deg}$ is an associative ideal;
\item $H_{\Deg}$ is a Zinbiel ideal, hence $H_{\Quandle}$ carries an induced Zinbiel product.
\end{itemize}
\end{thm}

The situation is rather subtle here. The Zinbiel product on rack cohomology does not restrict to the quandle cohomology; to get a Zinbiel product on $H_{\Quandle}$, we need to consider it as a quotient of $H_{\Rack}$. However, the associative product induced by the Zinbiel product does restrict to $H_{\Quandle}$.

\section{Quandle cohomology vs. rack cohomology}\label{S:Q2}

The rack cohomology of spindles and quandles shares a lot with the Hochschild cohomology of monoids and groups. This analogy suggests that the degenerate subcomplex $C^{\Deg}$ can be ignored, and that the rack cohomology $H_{\Rack}$ and the quandle cohomology $H_{\Quandle}$ carry the same information about a spindle. R.A. Litherland and S.~Nelson \cite{LiNe} showed this is not as straightforward as that: the degenerate part is highly non-trivial, and in particular contains the entire quandle part: 
\[C_\bullet^{\Deg} \simeq C_{\bullet-1}^{\Quandle} \oplus C_\bullet^{\Late}\ \text{ for } \bullet \geqslant 2.\]
Here $C_\bullet^{\Late}(X) := \ZZ X \otimes C_{\bullet-1}^{\Deg}(X)$ is the \emph{late degenerate subcomplex}, which is the linear envelope of all monomials with repetition at some place other than the beginning. This refines the rack cohomology splitting from Theorem~\ref{thm:Q}:
\begin{equation}\label{E:CohomDecompLate}
H^\bullet_{\Rack} \simeq H^\bullet_{\Quandle} \oplus H^{\bullet-1}_{\Quandle} \oplus H^\bullet_{\Late}\ \text{ for } \bullet \geqslant 2.
\end{equation}
We will now recover this decomposition in our bialgebraic setting. However, our methods are not sufficient for coupling this decomposition with the algebraic structure on $H_{\Rack}$:
\begin{question}\label{Qu:CohomDecompLate}
Do the cup product and the Zinbiel product on the rack cohomology of a spindle respect the decomposition~\eqref{E:CohomDecompLate} in any sense? In particular, can the quandle cohomology regarded as a \emph{Zinbiel algebra} be reconstructed from the degenerate cohomology?
\end{question}
Now, even though $H_{\Deg}$ is big, it is degenerate in a certain sense. Indeed,  J.~Przytycki and K.~Putyra \cite{PrzPutyra} showed the quandle cohomology $H_{\Quandle}$ of a spindle to completely determine its rack cohomology $H_{\Rack}$, and hence $H_{\Deg}$, on the level of abelian groups. In light of the preceding section, the following question becomes particularly interesting:

\begin{question}\label{Qu:ZinbielDegenerate}
Can Zinbiel and associative structures on the rack cohomology of a spindle be recovered from the corresponding structures on its quandle cohomology? 
\end{question}

\medskip

Let us now return to our d.g. bialgebra $B(X)$. 

\begin{prop}\label{P:late}
Let $X$ be a spindle. Put 
\[B^{\Late}(X):=B^+(X) \otimes_{A(X)} B^{\Deg}(X),\]
where $B^+$ is the positive degree part of $B$. One has the following isomorphism of graded $A(X)$-bimodules:
\begin{equation}\label{E:QuInsideDeg}
B_{\bullet}^{\Deg}(X) \simeq B_{\bullet}^{\Late}(X) \oplus B_{\bullet-1}^{\Quandle}(X)\ \text{ for } \bullet \geqslant 2.
\end{equation}
\end{prop}

This results immediately from the following technical lemma:

\begin{lem}\label{L:late}
Let $X$ be a spindle. Define a map $s \colon B^+(X) \to B^{\Deg}(X)$ as follows: take any element from $B^+$ written using the generators of the form $x$ and $e_y$, and in each of its monomials replace the first letter of the form $e_y$ by $e_ye_y$. Then:
\begin{itemize}
\item $s$ is a well-defined injective $A$-bilinear map of degree $1$;
\item one has the following decomposition of graded $A$-bimodules:
\[B^{\Deg} = B^{\Late} \oplus s((B^{\Norm})^+).\]
\end{itemize}
\end{lem}


The map $s$ yields the first degeneracy $s_1$ for the cubical structure underlying quandle cohomology, hence the notation.

\begin{proof}
To show that $s$ is well defined, one needs to check that it is compatible with the relation $e_xy = ye_{x^y}$ in $B$, that is, we should have $e_xe_xy = ye_{x^y}e_{x^y}$. This is indeed true:
\[e_xe_xy = e_x ye_{x^y} = ye_{x^y}e_{x^y}.\]
This map is $A$-bilinear and of degree $1$ by construction. Injectivity becomes clear if one writes all the monomials in $B$ in the canonical form $x_1\cdots x_k e_{y_1} \cdots e_{y_n}$.

Further, the map $b \otimes b' \mapsto bb'$ identifies $B^{\Late}$ with an $A$-invariant subspace of $B^{\Deg}$, which is again clear using canonical forms.

Now, $s(B^+)$ and $B^{\Late}$ are graded $A$-sub-bimodules of $B^{\Deg}$, with $s(B^+) + B^{\Late}=B^{\Deg}$ and $s(B^+) \cap B^{\Late}=s(B^{\Deg})$; as usual, this is clear using canonical forms. Since $s(B^+) =s(B^{\Deg} \oplus (B^{\Norm})^+) =s(B^{\Deg}) \oplus s((B^{\Norm})^+)$, we obtain the desired decomposition $B^{\Deg}=B^{\Late} \oplus s((B^{\Norm})^+)$.
\end{proof}

As usual, decomposition~\eqref{E:QuInsideDeg} implies
\begin{equation}\label{E:QuInsideDeg2}
\oB_{\bullet}^{\Deg}(X) \simeq \oB_{\bullet}^{\Late}(X) \oplus \oB_{\bullet-1}^{\Quandle}(X)\ \text{ for } \bullet \geqslant 2,
\end{equation}
with obvious notations. And as usual, this decomposition respects more structure than \eqref{E:QuInsideDeg}:

\begin{prop}\label{P:late2}
Let $X$ be a spindle. Then \eqref{E:QuInsideDeg2} is an isomorphism of differential graded $A(X)$-bimodules. 
\end{prop}

\begin{proof}
Let us check that the differential preserves $\oB^{\Late}$. An element of $\oB^{\Late}$ is a linear combination of (the $\sim$-equivalence classes of) terms of the form $e_xb$, with $b \in B^{\Deg}$. Since $d(e_xb)= -e_xd(b)$ in $\oB$ (relation~\eqref{E:BBar}) and since $d$ preserves $B^{\Deg}$, we have $d(e_x b) \in \oB^{\Late}$.

Now, the map $s$ from Lemma~\ref{L:late} induces a map $s \colon \oB^+ \to \oB^{\Deg}$. We need to check that the differential preserves $s((\oB^{\Norm}_{\bullet})^+) \simeq \oB_{\bullet-1}^{\Quandle}(X)$. Adapting the arguments from the proof of Proposition~\ref{P:Q1}, we see that an element of $s((\oB^{\Norm})^+)$ is a linear combination of the classes of terms of the form $(e_x^2-e_y^2)b$, with $b \in B^{\Norm}$. Then \eqref{E:dBD} yields
\begin{align*}
d((e_x^2-e_y^2)b)&=d(e_x^2-e_y^2)b+(e_x^2-e_y^2)d(b)=(e_x^2-e_y^2)d(b).
\end{align*}
Since $d$ preserves $B^{\Norm}$, we conclude $d((e_x^2-e_y^2)b) \in s((B^{\Norm})^+)$.
\end{proof}

\footnotesize
\bibliographystyle{alpha}
\bibliography{biblio}

\newcommand{\etalchar}[1]{$^{#1}$}
\def\cprime{$'$}
\begin{thebibliography}{CJK{\etalchar{+}}03}

\bibitem[AG03]{AG}
Nicol{\'a}s Andruskiewitsch and Mat{\'{\i}}as Gra{\~n}a.
\newblock From racks to pointed {H}opf algebras.
\newblock {\em Adv. Math.}, 178(2):177--243, 2003.

\bibitem[CES04]{HomologyYB}
J.~Scott Carter, Mohamed Elhamdadi, and Masahico Saito.
\newblock Homology theory for the set-theoretic {Y}ang--{B}axter equation and
  knot invariants from generalizations of quandles.
\newblock {\em Fund. Math.}, 184:31--54, 2004.

\bibitem[CJK{\etalchar{+}}03]{QuandleHom}
J.~Scott Carter, Daniel Jelsovsky, Seiichi Kamada, Laurel Langford, and
  Masahico Saito.
\newblock Quandle cohomology and state-sum invariants of knotted curves and
  surfaces.
\newblock {\em Trans. Amer. Math. Soc.}, 355(10):3947--3989, 2003.

\bibitem[Cla11]{Cl}
Frans Clauwens.
\newblock The algebra of rack and quandle cohomology.
\newblock {\em J. Knot Theory Ramifications}, 20(11):1487--1535, 2011.

\bibitem[Cov12]{C1}
Simon Covez.
\newblock On the conjectural {L}eibniz cohomology for groups.
\newblock {\em J. K-Theory}, 10(3):519--563, 2012.

\bibitem[Cov14]{C2}
Simon Covez.
\newblock {Rack homology and conjectural Leibniz homology}.
\newblock {\em ArXiv e-prints}, February 2014.

\bibitem[Deh00]{DehBook}
Patrick Dehornoy.
\newblock {\em Braids and self-distributivity}, volume 192 of {\em Progress in
  Mathematics}.
\newblock Birkh\"auser Verlag, Basel, 2000.

\bibitem[EN15]{QuandlesIntro}
Mohamed Elhamdadi and Sam Nelson.
\newblock {\em Quandles---an introduction to the algebra of knots}, volume~74
  of {\em Student Mathematical Library}.
\newblock American Mathematical Society, Providence, RI, 2015.

\bibitem[FGG16]{FG}
Marco~A. Farinati and Juliana Garc{\'{\i}}a~Galofre.
\newblock A differential bialgebra associated to a set theoretical solution of
  the {Y}ang--{B}axter equation.
\newblock {\em J. Pure Appl. Algebra}, 220(10):3454--3475, 2016.

\bibitem[FRS95]{FRS95}
Roger Fenn, Colin Rourke, and Brian Sanderson.
\newblock Trunks and classifying spaces.
\newblock {\em Appl. Categ. Structures}, 3(4):321--356, 1995.

\bibitem[FRS07]{FRS07}
Roger Fenn, Colin Rourke, and Brian Sanderson.
\newblock The rack space.
\newblock {\em Trans. Amer. Math. Soc.}, 359(2):701--740, 2007.

\bibitem[Joy82]{J82}
David Joyce.
\newblock A classifying invariant of knots, the knot quandle.
\newblock {\em J. Pure Appl. Algebra}, 23(1):37--65, 1982.

\bibitem[Kin07]{Kinyon}
Michael~K. Kinyon.
\newblock Leibniz algebras, {L}ie racks, and digroups.
\newblock {\em J. Lie Theory}, 17(1):99--114, 2007.

\bibitem[Leb13]{Lebed1}
Victoria Lebed.
\newblock Homologies of algebraic structures via braidings and quantum
  shuffles.
\newblock {\em J. Algebra}, 391:152--192, 2013.

\bibitem[Leb17]{L16}
Victoria Lebed.
\newblock Cohomology of idempotent braidings with applications to factorizable
  monoids.
\newblock {\em Internat. J. Algebra Comput.}, 27(4):421--454, 2017.

\bibitem[LN03]{LiNe}
R.~A. Litherland and Sam Nelson.
\newblock The {B}etti numbers of some finite racks.
\newblock {\em J. Pure Appl. Algebra}, 178(2):187--202, 2003.

\bibitem[Mat82]{M82}
Sergei Matveev.
\newblock Distributive groupoids in knot theory.
\newblock {\em Mat. Sb. (N.S.)}, 119(161)(1):78--88, 160, 1982.

\bibitem[PP13]{PrzPuLattices}
J\'{o}zef~H. Przytycki and Krzysztof~K. Putyra.
\newblock Homology of distributive lattices.
\newblock {\em J. Homotopy Relat. Struct.}, 8(1):35--65, 2013.

\bibitem[PP16]{PrzPutyra}
J\'ozef~H. Przytycki and Krzysztof~K. Putyra.
\newblock The degenerate distributive complex is degenerate.
\newblock {\em Eur. J. Math.}, 2(4):993--1012, 2016.

\bibitem[Prz11]{Prz1}
J{\'o}zef~H. Przytycki.
\newblock Distributivity versus associativity in the homology theory of
  algebraic structures.
\newblock {\em Demonstratio Math.}, 44(4):823--869, 2011.

\bibitem[Prz15]{PrzSurvey}
J{\'o}zef~H. Przytycki.
\newblock Knots and distributive homology: from arc colorings to
  {Y}ang--{B}axter homology.
\newblock In {\em New ideas in low dimensional topology}, volume~56 of {\em
  Ser. Knots Everything}, pages 413--488. World Sci. Publ., Hackensack, NJ,
  2015.

\bibitem[Ron00]{Ronco}
Mar\'\i~a Ronco.
\newblock Primitive elements in a free dendriform algebra.
\newblock In {\em New trends in {H}opf algebra theory ({L}a {F}alda, 1999)},
  volume 267 of {\em Contemp. Math.}, pages 245--263. Amer. Math. Soc.,
  Providence, RI, 2000.

\bibitem[Ros95]{R95}
Marc Rosso.
\newblock Groupes quantiques et alg\`ebres de battage quantiques.
\newblock {\em C. R. Acad. Sci. Paris S\'er. I Math.}, 320(2):145--148, 1995.

\bibitem[Ser51]{SerreThesis}
Jean-Pierre Serre.
\newblock Homologie singuli\`ere des espaces fibr\'es. {A}pplications.
\newblock {\em Ann. of Math. (2)}, 54:425--505, 1951.

\bibitem[Szy18]{Szymik}
Markus Szymik.
\newblock Permutations, power operations, and the center of the category of
  racks.
\newblock {\em Comm. Algebra}, 46(1):230--240, 2018.

\bibitem[Tak43]{T43}
Mituhisa Takasaki.
\newblock Abstractions of symmetric functions.
\newblock {\em Tohoku Math. J.}, 49:143--207, 1943.

\end{thebibliography}
\end{document}